\documentclass[8pt]{article}
\usepackage{amsmath,amssymb,amsbsy,amsfonts,amsthm,latexsym,
            amsopn,amstext,amsxtra,euscript,amscd,amsthm,url}
\usepackage[usenames, dvipsnames]{color}
\usepackage[utf8]{inputenc}
\usepackage[english]{babel}
\usepackage{mathrsfs}
\usepackage{caption}

\newtheorem{thm}{Theorem}
\newtheorem{rem}{Remark}

\newtheorem{lem}{Lemma}
\newtheorem{cor}{Corollary}

\DeclareMathOperator{\1}{\textbf{1}}

\DeclareMathOperator{\id}{id}
\global\long\def\epsilon{\varepsilon} 
\usepackage{amsmath}
\DeclareMathAlphabet{\mathpzc}{OT1}{pzc}{m}{it}

\begin{document}
\date{}

\title{\bf On  sums of logarithmic  averages of gcd-sum functions}
\author{Isao Kiuchi and Sumaia Saad Eddin}

\maketitle
{\def\thefootnote{}
\footnote{{\it Mathematics Subject Classification 2010: 11A25, 11N37, 11Y60.}} 

\begin{abstract}
Let 
$
\gcd(k,j) 
$
be the greatest common divisor of the integers $k$ and $j$. 
For any arithmetical function $f$,  we   establish  several  asymptotic formulas  for  weighted  averages  
of gcd-sum functions  with  weight concerning logarithms,  
that is  
$$
\sum_{k\leq x}\frac{1}{k} \sum_{j=1}^{k}f(\gcd(k,j)) \log j. 
$$
More precisely, we  give  asymptotic formulas  for various multiplicative functions 
such as $f=\id$,  $\phi$,  $\id_{1+a}$ and $\phi_{1+a}$ with $-1<a<0$.  We also establish some formulas of 
Dirichlet series having coefficients of the sum function 
$
\sum_{j=1}^{k}s_{k}(j)\log j 
$ where $s_{k}(j)$ is Anderson--Apostol sums. 
\end{abstract}


                               \section{Introduction and main results}


Let  $\gcd(k,l)$ be  the greatest common divisor of the integers $k$ and $l$, and  let  
 $\mu$   be the M\"obius function.  
It is known that the Ramanujan sum is defined by 
$$ 
c_{k}(j)=\sum_{d|\gcd(j,k)}d\mu\left(\frac{k}{d}\right). 
$$ 
We recall that the Dirichlet convolution $f * g$ of the arithmetical functions $f$ and $g$ is defined by  $(f * g)(n) = \sum_{d \mid n} f(d) g\left({n}/{d}\right)$  
for any positive integer $n$. Let  $\Lambda$  be  the von Mangoldt function defined by 
$
\Lambda = \mu*\log.
$ 
The weighted average of the Ramanujan sum with weight concerning logarithms
was first established by T\'{o}th~\cite{To}. For any positive integer $k$, he derived the following interesting identity
\begin{align}                                       \label{toth}
\frac{1}{k}\sum_{j=1}^{k}c_{k}(j)\log j 
&= \Lambda(k) + \sum_{d|k}\frac{\mu(d)}{d}\log d!. 
\end{align} 
There is in the literature a large number of generalizations of $c_{k}(j)$.  The following sum is one of the most common generalization of the Ramanujan sum, and due to Anderson-Apostol, see~\cite{AA}, 
$$
s_{k}(j) := \sum_{d|\gcd(k,j)}f(d)g\left(\frac{k}{d}\right),
$$
for any positive integers $k$ and $j$  and  any arithmetical functions $f$ and $g$.\\

We define the arithmetic functions $\1(n)$ and $\id (n)$ by $\1(n) = 1$ 
and $\id (n)=n$ respectively, for any positive integer  $n$. The arithmetical function $\id_{a}$ is given by  $\id_{a}(n)=n^a$ for any real number $a$.
In case of $a=1$, we write $\rm id_1=\id$. 
The  weighted average of  $s_{k}(j)$  with weight concerning  logarithms  was  first  studied by the first author, Minamide and  Ueda~\cite{KMU}. 
They showed that 
\begin{equation}                                                                                   \label{KMU}
\sum_{j=1}^{k}s_{k}(j)\log j = (f\cdot \log * g\cdot \id)(k) + (f*g\cdot L)(k),
\end{equation}
 where
$
L(d)=\sum_{m=1}^{d}\log m.     
$ 
That is a generalization of \eqref{toth}. Define
 \begin{equation*}
K(x;f,g) :=  \sum_{k\leq x}\frac{1}{k} \sum_{j=1}^{k}s_{k}(j) \log j,
 \end{equation*}
 for  any positive real number $x>1$. 
The first  purpose of this paper is to study $K(x;f,g)$ and prove that:


 \begin{thm}                                                                                     \label{th0}
Let $f$ and $g$ be  any arithmetical functions. 
There is a  certain  positive  constant   $\Theta$  
such that   
\begin{align}                                                                                 \label{sum_thK}
 K(x;f,g)  
=  \sum_{n\leq x} & \frac{(f*g\cdot \id)(n)}{n} \log \frac{n}{e} 
+  \frac{1}{2} \sum_{n\leq x}\frac{(f*g\cdot \log )(n)}{n}  \nonumber \\  
&+ \log\sqrt{2\pi} \sum_{n\leq x}\frac{(f*g)(n)}{n}          
+  {\Theta} \sum_{n\leq x}\frac{(f* g\cdot \id_{-1})(n)}{n}
\end{align} 
for  any positive real number $x>1$. 
\end{thm}
The arithmetical functions  $\tau(n)$ and $\sigma (n)$   denote 
 the number and sum of the positive  divisors of $n$,   respectively. 
Let $\sigma_a$ be the generalized divisor function, for any real number $a$, defined by $\1*\id_a$. 
As a special case of Theorem~\ref{th0}, we take $g=f$ where $f$ is a completely  multiplicative. Then, we deduce immediately that:


\begin{cor}                                                                                \label{corol1}
Let $f$ be a completely multiplicative function, and let $l$ be a modified divisor function defined by  $l:=\1 *\log$. Under the hypotheses of Theorem~\ref{th0}, we have
\begin{multline}                                                                            \label{sum_thKL1}
K(x;f,f) 
= \sum_{n\leq x} \frac{f(n)\sigma(n)}{n}  \log \frac{n}{e} 
+  \frac{1}{2} \sum_{n\leq x}\frac{f(n)l(n)}{n}  \\
+ \log\sqrt{2\pi} \sum_{n\leq x} \frac{f(n)\tau(n)}{n}       
+  {\Theta} \sum_{n\leq x}\frac{f(n)\sigma_{-1}(n)}{n}.    
\end{multline}
\end{cor}
Now, we let $f=\1$ in the above. Then, we have: 


\begin{cor}                                                                            \label{corol2}
For any  positive real number $x>1$, we have
\begin{multline}                                      
\label{sum_thKL12}     
\sum_{k\leq x}\frac{1}{k}\sum_{j=1}^{k}\tau(\gcd(k,j)) \log j  
= \zeta(2)x\log x-2\zeta(2)x + \frac{1}{12}(\log x)^3 \\
+ \frac{\gamma-1 + \log 2\pi}{4}  (\log x)^2 + O\left((\log x)^{5/3}\right),                    
\end{multline}
where $\zeta$ be the Riemann zeta-function and $\gamma$ is the Euler constant. 
\end{cor}

Let $\phi$  be  the Euler totient function defined by 
$
\phi=\id * \mu,   
$
and let $\phi_{a}$  be the Jordan totient function  defined by 
$
\phi_a=\mu*\id_a
$
, for any real number $a$.
In 1885,  Ces\'{a}ro~\cite{C} proved the  well-known  identity  
$$
\sum_{j=1}^{k}f(\gcd(j,k)) = (f*\phi)(k)    
$$
 for any positive integer $k$  and any arithmetical function $f$.  
 From this latter, one  can easily obtain  the formula   
\begin{equation}
\label{Cesaro}                            
Y(x;f):=\sum_{k\leq x}\frac{1}{k}\sum_{j=1}^{k}f(\gcd(j,k)) 
= \sum_{n\leq x}\frac{(f*\phi)(n)}{n},           
\end{equation}
for any positive number $x>1$.
Define  
$$
L^{}(x;f) 
: = \sum_{k\leq x} \frac{1}{k} \sum_{j=1}^{k}f(\gcd(k,j)) \log j.   
$$
The second purpose of this paper is to give an identity of $L(x;f)$. We prove that:


 \begin{thm}                                                                                        \label{th1}
Let $f$ be an arithmetical function. There is a certain positive constant $\Theta$ such that   
\begin{multline}                                   \label{sum_thL}
L(x;f) 
= \sum_{n\leq x} \frac{(f*\phi)(n)}{n}  \log \frac{n}{e} 
+  \frac{1}{2} \sum_{n\leq x}\frac{(f*\Lambda)(n)}{n}    \\
+ \log\sqrt{2\pi} \sum_{n\leq x} \frac{f(n)}{n}           
+  {\Theta} \sum_{n\leq x}\frac{(f* \phi_{-1})(n)}{n}  
\end{multline}
for any positive real number $x>1$. 
\end{thm}

\begin{rem}                                                                                     \label{rem1}
Suppose that  the first term on the right-hand side  of (\ref{sum_thL}) is  the main term.  
Using \eqref{Cesaro} and  the partial summation,  we  have  the relation  
$$
\sum_{n\leq x} \frac{(f*\phi)(n)}{n}  \log \frac{n}{e} 
= Y(x;f)\log\frac{x}{e} - \int_{1}^{x}Y(u;f)\frac{du}{u}.
$$
 Then  the   order of magnitude for the function
$L(x;f)$  may be regarded as 
$$
L(x;f) \asymp Y(x;f)\log\frac{x}{e} - \int_{1}^{x}Y(u;f)\frac{d u}{u}.
$$
\end{rem}
Replacing $f$ into  \eqref{sum_thL}  by $f*\1$  and  then by $f*\tau$,   
we obtain the following asymptotic formulas,  namely


\begin{cor}                                                                           \label{corol3}
We have 
\begin{multline}                                                                      \label{sum_thL1}
L(x;f*\1) 
= \sum_{n\leq x} \frac{(f*\id)(n)}{n}  \log \frac{n}{e} 
+  \frac{1}{2} \sum_{n\leq x}\frac{(f*\log)(n)}{n} \\
+ \log\sqrt{2\pi} \sum_{n\leq x} \frac{(f*\1)(n)}{n}           
+  {\Theta} \sum_{n\leq x}\frac{(f* \id_{-1})(n)}{n},   
\end{multline}
and  
\begin{multline}                                                                   \label{sum_thL2}
L(x;{f*\tau}) 
= \sum_{n\leq x} \frac{(f*\sigma)(n)}{n}  \log \frac{n}{e} 
+  \frac{1}{2} \sum_{n\leq x}\frac{(f*l)(n)}{n} \\
+ \log\sqrt{2\pi} \sum_{n\leq x} \frac{(f*\tau)(n)}{n}           
+  {\Theta} \sum_{n\leq x}\frac{(f*\sigma_{-1})(n)}{n},    
\end{multline}
where $l=\1*\log$.   
\end{cor}
The proof of the theorems above are  not difficult, 
but the feature of them is  that they provide many interesting and useful formulas, which are given in Section~\ref{section2}. In Section~\ref{section3}, we establish some formulas of Dirichlet series having coefficients with partial sums for weighted averages of $s_k(j).$

\section{Applications of Theorems~\ref{th0} and \ref{th1}}
\label{section2}
Taking  $f=\id$ and $g=\mu$ in~\eqref{sum_thK}, one can obtain   
$$ 
K(x;\id,\mu) =\sum_{k\leq x}\frac{1}{k}\sum_{J=1}^{k}c_{k}(j)\log j. 
$$
In this section, we show that


\begin{thm}                                                                                \label{corol0}
Under the hypotheses of Theorem \ref{th0}, we have
\begin{align}                                                                            \label{sum_coroK}
\sum_{k\leq x}\frac{1}{k}\sum_{J=1}^{k}c_{k}(j)\log j  
&=  \left(\frac{\log\sqrt{2\pi}}{\zeta(2)} + \frac{\zeta'(2)}{2\zeta^2(2)}  + \frac{\Theta}{\zeta(3)}\right)x 
+   O\left((\log x)^2\right).
\end{align}
\end{thm}
We recall that 
$$
\sum_{n\leq x}\tau(n) = x\log x + \left(2\gamma -1\right)x + \Delta(x)  
$$
with the error term  $\Delta(x)=O\left(x^{\theta+\epsilon}\right)$ for any small number  $\epsilon>0$,
and  $1/4\leq \theta \leq 131/416$. The upper bound has been given by Huxley \cite{H}. 
It is known that the inequality 
\begin{align}                                                                        \label{Voronoi}
\int_{1}^{X}\Delta(y)dy &= O\left(X\right) 
\end{align} 
holds for any positive number $X$.
We shall provide applications of Theorem~\ref{th1}, 
for various multiplicative functions such as $f=\id$,  $\phi$,  $\id_{1+a}$ and $\phi_{1+a}$.  
Other functions as $\psi, \psi_{1+a}, \phi^2$ and $\psi^2$, where $\psi$ and $\psi_{1+a}$ denote the Dedekind function and its generalization respectively, can be considered too.\\
We prove that: 

\begin{thm}                                                                             \label{th4}
Let the notation be as above.  
For any  sufficiently large  positive  number $x>1$, we have  
\begin{align}                                                                 \label{sum_cor11}
L(x;\id) 
&= \frac{1}{\zeta(2)}x(\log x)^2
+ \frac{1}{\zeta(2)}\left(2\gamma - 3 -\frac{\zeta'(2)}{\zeta(2)}\right)x\log x \nonumber \\
&- \frac{1}{\zeta(2)}\left(4\gamma -3 -2\frac{\zeta'(2)}{\zeta(2)} 
+ \frac{\zeta'(2)}{2} - {\Theta\zeta(3)} - \zeta(2)\log\sqrt{2\pi} \right)x    \nonumber \\ 
& + \sum_{n\leq x}\frac{\mu(n)}{n}\Delta\left(\frac{x}{n}\right)\log\frac{x}{e} + O\left((\log x)^2\right),    
\end{align}
and 
\begin{align}                                                                   \label{sum_cor12}
L(x;\phi) 
&= \frac{1}{\zeta^{2}(2)}x(\log x)^2 
+ \frac{1}{\zeta^{2}(2)}\left(2\gamma - 3 -2\frac{\zeta'(2)}{\zeta(2)}\right)x\log x \nonumber \\
&- \frac{1}{\zeta^{2}(2)}\left(4\gamma -3 -4\frac{\zeta'(2)}{\zeta(2)}  
+ \frac{\zeta'(2)}{2} - {\Theta\zeta(3)}  - \zeta(2)\log\sqrt{2\pi}\right)x    \nonumber \\
& + \sum_{n\leq x}\frac{(\mu*\mu)(n)}{n}\Delta\left(\frac{x}{n}\right)\log\frac{x}{e} 
+ O\left((\log x)^3\right).     
\end{align}
\end{thm}
For $-1<a<0$, we recall that 
$$
\sum_{n\leq x}\sigma_{a}(n) = \zeta(1-a)x + \frac{\zeta(1+a)}{1+a}x^{1+a} - \frac{\zeta(-a)}{2}  + \Delta_{a}(x). 
$$
Here the error term $\Delta_{a}(x)$ is given by
\begin{equation}                                                                              \label{eq001} 
\Delta_{a}(x) = 
\frac{x^{\frac14 +\frac{a}{2}}}{\pi\sqrt{2}}
\sum_{n\leq N}\frac{\sigma_{a}(n)}{n^{\frac{3}{4}+\frac{a}{2}}}\cos\left(4\pi\sqrt{nx}-\frac{\pi}{4}\right)
+ O\left(x^{\frac12 + \varepsilon}N^{-\frac12}\right),  
\end{equation}
for  $1\ll N\ll x$, see~\cite[Eq. (64)]{KS}.  From this latter,  we can see that
\begin{equation}                                                                          
\label{Voronoia}
\int_{1}^{X}\Delta_{a}(y)dy = O_{a}\left(X^{3/4+a/2}\right),
\end{equation} 
for any positive number $X$.
\begin{thm}                                                                                 \label{th5}
Let the notation be as above.  
For any  sufficiently large  positive  number $x>1$  and $-1<a<0$, we have  
\begin{align}                                                                 \label{sum_cor21}
L(x;\id_{1+a}) 
&=  \frac{\zeta(1-a)}{\zeta(2)}x\log^{}x 
-  \frac{2\zeta(1-a)}{\zeta(2)} x 
+  \frac{\zeta(1+a)}{(1+a)\zeta(2+a)} x^{1+a}\log x  \nonumber \\
&- \frac{1}{(1+a)\zeta(2+a)}\left(\frac{(2+a)\zeta(1+a)}{1+a} 
 + \frac{\zeta'(2+a)}{2} 
 - {\Theta\zeta(3+a)}\right)x^{1+a}         \nonumber \\
& + \frac{\log\sqrt{2\pi}}{1+a}x^{1+a} 
+  \sum_{n\leq x}\frac{\mu(n)}{n}\Delta_{a}\left(\frac{x}{n}\right)\log\frac{x}{e} 
+  O_{a}\left(\log x\right),     
\end{align}
and 
\begin{align}                                                                       
\label{sum_cor22}
L(x;\phi_{1+a}) 
&= \frac{\zeta(1-a)}{\zeta^2(2)}x\log^{}x 
-  \frac{2\zeta(1-a)}{\zeta^2(2)} x 
+  \frac{\zeta(1+a)}{(1+a)\zeta^2(2+a)} x^{1+a}\log x  \nonumber \\
&- \frac{1}{(1+a)\zeta^2(2+a)}\left(\frac{(2+a)\zeta(1+a)}{(1+a)} 
                                + \frac{\zeta'(2+a)}{2} 
 - {\Theta\zeta(3+a)}\right)x^{1+a}         \nonumber \\
& + \frac{\log\sqrt{2\pi}}{(1+a)\zeta(2+a)}x^{1+a} 
+  \sum_{n\leq x}\frac{(\mu*\mu)(n)}{n}\Delta_{a}\left(\frac{x}{n}\right)\log\frac{x}{e} 
+  O_{a}\left((\log x)^3\right).     
\end{align}
\end{thm}


\begin{rem}  
From Theorems \ref{th4} and \ref{th5}, we deduce that 
$$
\lim_{x\to \infty}\frac{L^{}(x;\id)}{x\log^{2}x}=\frac{1}{\zeta(2)}  \qquad  {\rm and} \qquad  
 \lim_{x\to \infty}\frac{L^{}(x;\phi)}{x\log^{2}x}=\frac{1}{\zeta^{2}(2)}, 
$$
and that 
$$
\lim_{x\to \infty}\frac{L^{}(x;\id_{1+a})}{x\log^{}x}=\frac{\zeta(1-a)}{\zeta(2)}  
\qquad {\rm and} \qquad  
 \lim_{x\to \infty}\frac{L^{}(x;\phi_{1+a})}{x\log^{}x}=\frac{\zeta(1-a)}{\zeta^{2}(2)}.  
$$
\end{rem}
\section{Dirichlet series}
\label{section3}
Given two functions  $F(s)$ and $G(s)$  represented by  Dirichlet series as follows:   
$$
F(s)=\sum_{k=1}^{\infty}\frac{f(k)}{k^s} \qquad \Re (s)>\sigma_1, 
$$
and 
$$
G(s)=\sum_{k=1}^{\infty}\frac{g(k)}{k^s} \qquad \Re (s)>\sigma_2,
$$
which converge   absolutely in the  half-plane $\Re (s)>\sigma_1$ and $\Re (s)>\sigma_2$ respectively.   
The Dirichlet series of the first derivative of $F(s)$ and $G(s)$, 
with respect to $s$,  are given by 
$$
F'(s)=- \sum_{k=1}^{\infty}\frac{f(k)}{k^s}\log k  \qquad \Re (s)>\sigma_1,
$$
and 
$$
G'(s)=- \sum_{k=1}^{\infty}\frac{g(k)}{k^s}\log k   \qquad \Re (s)>\sigma_2.  
$$
We  shall consider  the relationship between Dirichlet series having the coefficients with partial sums for weighted averages of $s_{k}(j)$  and two  Dirichlet series $F(s)$ and $G(s)$.
Define 
$$
u_{f,g}(k) := \sum_{j=1}^{k}s_{k}(j) \log j 
$$
for any positive integer $k$.   
Then, the Dirichlet series having the coefficients  $u_{f,g}(k)$  is defined by  
$$
U_{f,g}(s) := \sum_{k=1}^{\infty}\frac{u_{f,g}(k)}{k^s},   
$$
which  
converges absolutely in the region  $\Re (s)>\alpha$. From \eqref{KMU}, we easily  deduce  that:   


\begin{thm}                                                                                        \label{th7}
Let the notation be as above. Then  we  have
\begin{align}                                                                                  \label{sum_th7}
U_{f,g}(s) 
=  -F'(s)G(s-1) + F(s)G_{L}(s), 
\end{align}
where $G_{L}(s)$ is defined by  
$$
G_{L}(s):=\sum_{k=1}^{\infty}\frac{g(k)L(k)}{k^s}, 
$$
and converges absolutely  in the  half-plane  $\Re (s)>\sigma_2 +1$.   
\end{thm}

Here is an application of  \eqref{sum_th7}.  


\begin{cor}                                                                                        \label{th9}
Let the notation be as above. There is a  certain  positive  constant   $\Theta$   
such that 
\begin{multline}                                                                                  \label{sum_th91}
U_{f,g}(s) 
=  -F'(s)G(s-1) - F(s)\left(G'(s-1) +G(s-1)\right) 
\\+ F(s)\left(-\frac12 G'(s)  + \log\sqrt{2\pi} G(s)  
+ \Theta G(s+1)\right),   
\end{multline}
where $\Re (s)> \max\{\sigma_{1},\sigma_{2}+1\}$,
and that 
\begin{multline}                                                                                  \label{sum_th92}
U_{f*\mu,\1}(s) 
=  \frac{(F(s)\zeta'(s)- F'(s)\zeta(s))\zeta(s-1)}{\zeta^{2}(s)} 
- \left(\zeta(s-1) + \zeta'(s-1)\right)\frac{F(s)}{\zeta(s)}\\
-  \frac{F(s)}{2}\frac{\zeta'(s)}{\zeta(s)}  + \log\sqrt{2\pi} F(s)  + \Theta \frac{F(s)}{\zeta(s)}{\zeta(s+1)},
\end{multline} 
in the region $\Re (s)> \max\{\sigma_{1},2\}$. 
\end{cor}
As a consequence of the above,we immediately get the following formulas:
\begin{multline}                                                                                  \label{sum_th1001}
U_{\id,\mu}(s) 
=  \frac{\zeta(s-1)\zeta'(s)}{2\zeta^{2}(s)} + \log\sqrt{2\pi}\frac{\zeta(s-1)}{\zeta(s)} 
 + \Theta \frac{\zeta(s-1)}{\zeta(s+1)} 
 \\-\frac{\zeta'(s-1)}{\zeta(s-1)} + \frac{\zeta(s)}{\zeta(s-1)}\left(\frac{\zeta'(s-1)}{\zeta(s-1)}-1\right),
\end{multline}
and 
\begin{multline}                                                                            \label{sum_th1002}
U_{\phi,\1}(s) 
=  \frac{(\zeta(s-1)\zeta'(s)- \zeta'(s-1)\zeta(s))\zeta(s-1)}{\zeta^{2}(s)} 
 - \left(\zeta(s-1) + \zeta'(s-1)\right)\frac{\zeta(s-1)}{\zeta(s)}  \\ 
-  \frac{\zeta(s-1)}{2}\frac{\zeta'(s)}{\zeta(s)}       
 + \log\sqrt{2\pi} \zeta(s-1)  + \Theta \frac{\zeta(s-1)}{\zeta(s)}{\zeta(s+1)},
\end{multline}
where $\Re (s)> 2$. 

\section{Auxiliary results}


In order to prove Theorems~\ref{th4} and \ref{th5}, we prepare the next lemmas.





\begin{lem}                                                                     \label{lem10}
Let $\gamma$ denote the Euler constant. 
For any sufficiently large positive number $x>1$,  we have
\begin{multline}                                                               \label{lem110}
\sum_{n\leq x} \frac{(\id*\phi)(n)}{n}  \log \frac{n}{e}  = 
\frac{1}{\zeta(2)}x(\log x)^2 + \frac{1}{\zeta(2)}\left(2\gamma - 3 -\frac{\zeta'(2)}{\zeta(2)}\right)x\log x  \\
- \frac{1}{\zeta(2)}\left(4\gamma -3 -2\frac{\zeta'(2)}{\zeta(2)}\right)x + M(x),   
\end{multline}
and  
\begin{multline}                                                               \label{lem120}
\sum_{n\leq x} \frac{(\phi*\phi)(n)}{n} \log \frac{n}{e}  
= \frac{1}{\zeta^{2}(2)}x(\log x)^2 
+ \frac{1}{\zeta^{2}(2)}\left(2\gamma - 3 -2\frac{\zeta'(2)}{\zeta(2)}\right)x\log x  \\
- \frac{1}{\zeta^{2}(2)}\left(4\gamma -3 -4\frac{\zeta'(2)}{\zeta(2)}\right)x + \widetilde{P}(x),   
\end{multline}
where the functions $M(x)$ and $\widetilde{P}(x)$ are given by 
\begin{equation}     
\label{M}
M(x) = \sum_{n\leq x}\frac{\mu(n)}{n}\Delta\left(\frac{x}{n}\right)\log\frac{x}{e} + O\left((\log x)^2\right), 
\end{equation} 
and   
\begin{equation}                                                                \label{wideP}
\widetilde{P}(x) = \sum_{n\leq x}\frac{(\mu*\mu)(n)}{n}\Delta\left(\frac{x}{n}\right)\log\frac{x}{e} 
+ O\left((\log x)^3\right).   
\end{equation} 
\end{lem}


\begin{proof} 
Notice that 
$
\frac{\id*\phi}{\id}= \frac{\phi}{\id}*\1.   
$
 From  the identity  (2.10) in \cite{K1}, 
we have  
\begin{eqnarray}                                                                      \label{lem111}
\sum_{n\leq x} \frac{(\id*\phi)(n)}{n}  
&=&\frac{1}{\zeta(2)}x\log x + \frac{x}{\zeta(2)}\left(2\gamma -1 - \frac{\zeta'(2)}{\zeta(2)}\right) + E(x), 
\end{eqnarray}
where  
\begin{equation}                                                                             \label{E}    
E(x) = \sum_{n\leq x}\frac{\mu(n)}{n}\Delta\left(\frac{x}{n}\right) + O\left(\log x\right). 
\end{equation} 
We use the partial summation and  \eqref{lem111}   to get 
\begin{align}                                                                              \label{lem112}
& \sum_{n\leq x}\frac{(\id*\phi)(n)}{n}\log\frac{n}{e} \nonumber \\
&=\left(\frac{1}{\zeta(2)}x\log x
 + \frac{x}{\zeta(2)}\left(2\gamma -1 - \frac{\zeta'(2)}{\zeta(2)}\right) + E(x)\right)\log \frac{x}{e} 
 \nonumber \\
&- \int_{1}^{x}\left(\frac{1}{\zeta(2)}u\log u 
+ \frac{u}{\zeta(2)}\left(2\gamma -1 - \frac{\zeta'(2)}{\zeta(2)}\right) 
+ E(u)\right)\frac{d u}{u} \nonumber \\
&=\frac{1}{\zeta(2)}x(\log x)^2
+ \frac{1}{\zeta(2)}\left(2\gamma - 3 -\frac{\zeta'(2)}{\zeta(2)}\right)x\log x \nonumber \\
&- \frac{1}{\zeta(2)}\left(4\gamma -3 -2\frac{\zeta'(2)}{\zeta(2)}\right)x 
 + E(x)\log\frac{x}{e} - \int_{1}^{x}\frac{E(u)}{u}du + O\left(1\right). 
\end{align}
Using \eqref{Voronoi} and integrating by part, we derive the estimation
\begin{eqnarray}                                                 
\label{lem113}
\int_{1}^{x}\frac{E(u)}{u}du
 &=&   \sum_{n\leq x}\frac{\mu(n)}{n}\int_{n}^{x}\Delta\left(\frac{u}{n}\right)\frac{d u}{u} 
+ O\left(\int_{1}^{x}\frac{\log u}{u}du\right)   \nonumber \\
&=&  
\sum_{n\leq x}\frac{\mu(n)}{n}\int_{1}^{x/n}\Delta\left(y\right)\frac{dy}{y}  
+ O\left((\log x)^2\right) \nonumber \\
&=& O\left(\sum_{n\leq x}\frac{1}{n}\left(1+\int_{1}^{x/n} \frac{dy}{y}\right)\right)  
+ O\left((\log x)^2\right) \nonumber \\
&=& O\left((\log x)^2\right). 
\end{eqnarray}
From  \eqref{lem112}  and \eqref{lem113},  we complete  the proof of \eqref{lem110}. \\

We use the identity  
$
\frac{\phi*\phi}{\id}= \frac{\mu * \phi}{\id}*\1,
$ 
 and (2.11) in \cite{K1} 
to deduce  
\begin{align}                                                               \label{lem121}
\sum_{n\leq x} \frac{(\phi*\phi)(n)}{n} 
&=\frac{1}{\zeta^{2}(2)}x\log x + \frac{x}{\zeta^{2}(2)}\left(2\gamma -1 - 2 \frac{\zeta'(2)}{\zeta(2)}\right) 
+ P(x), 
\end{align}
where  
\begin{align}                                                      \label{P} 
P(x) = \sum_{n\leq x}\frac{(\mu*\mu)(n)}{n}\Delta\left(\frac{x}{n}\right) + O\left((\log x)^2\right).
\end{align} 
Using the partial summation and \eqref{lem121}, we find that  
\begin{multline}                                                                 \label{lem122}
  \sum_{n\leq x}  \frac{(\phi*\phi)(n)}{n}\log\frac{n}{e} 
=\frac{1}{\zeta^{2}(2)}x(\log x)^2 
+ \frac{1}{\zeta^{2}(2)}\left(2\gamma - 3 -2\frac{\zeta'(2)}{\zeta(2)}\right)x\log x 
\\- \frac{1}{\zeta^{2}(2)}\left(4\gamma -3 -4\frac{\zeta'(2)}{\zeta(2)}\right)x 
+ P(x)\log\frac{x}{e} - \int_{1}^{x}\frac{P(u)}{u}du + O\left(1\right). 
\end{multline}
Again, we use \eqref{Voronoi} to get
\begin{eqnarray}                                                               \label{lem123}
\int_{1}^{x}\frac{P(u)}{u}du
&= &   \sum_{n\leq x}\frac{(\mu*\mu)(n)}{n}\int_{1}^{x/n}\Delta\left(u\right)\frac{du}{u} 
 + O\left((\log x)^3\right) \nonumber \\
&=&O\left(\sum_{n\leq x}\frac{\tau(n)}{n}\left(1+\int_{1}^{x/n} \frac{du}{u}\right)\right) 
 + O\left((\log x)^3\right)   \nonumber \\
&=& O\left((\log x)^3\right). 
\end{eqnarray}
From \eqref{lem122}   and  \eqref{lem123}, we complete the proof of \eqref{lem120}.   
\end{proof}





\begin{lem}                                                                         
\label{lem20}
For any sufficiently large positive number $x>1$  and  $-1<a<0$,  we have
\begin{align}                                                               
\label{lem210}
\sum_{n\leq x} & \frac{(\id_{1+a}*\phi)(n)}{n}  \log \frac{n}{e}  
= \frac{\zeta(1-a)}{\zeta(2)}x\log^{}x - \frac{2\zeta(1-a)}{\zeta(2)}x   \nonumber \\
&+ \frac{\zeta(1+a)}{(1+a)\zeta(2+a)}x^{1+a}\log x
 - \frac{(2+a)\zeta(1+a)}{(1+a)^{2}\zeta(2+a)}x^{1+a} +  M_{a}(x),   
\end{align}
and  
\begin{align}                                                               \label{lem220}
\sum_{n\leq x}& \frac{(\phi_{1+a}*\phi)(n)}{n} \log \frac{n}{e}  
= \frac{\zeta(1-a)}{\zeta^{2}(2)}x\log^{}x  - \frac{2\zeta(1-a)}{\zeta^{2}(2)}x \nonumber \\
& + \frac{\zeta(1+a)}{(1+a)\zeta^{2}(2+a)}x^{1+a}\log x 
- \frac{(2+a)\zeta(1+a)}{(1+a)^{2}\zeta^{2}(2+a)} x^{1+a} 
  + \widetilde{P}_{a}(x),   
\end{align}
where the functions $M_{a}(x)$  and $\widetilde{P}_{a}(x)$ are   given by 
\begin{equation}                                                                  
\label{Ma}
M_{a}(x) = \sum_{n\leq x}\frac{\mu(n)}{n}\Delta_{a}\left(\frac{x}{n}\right)\log\frac{x}{e} 
+ O_{a}\left(\log x\right), 
\end{equation}  
and 
\begin{equation}                                                                \label{widePa}
\widetilde{P}_{a}(x) = \sum_{n\leq x}\frac{(\mu*\mu)(n)}{n}\Delta_{a}\left(\frac{x}{n}\right)\log\frac{x}{e} 
+ O_{a}\left((\log x)^3\right).   
\end{equation} 
\end{lem}


\begin{proof} 
Using the fact that   
$
\frac{\id_{1+a}*\phi}{\id}=\frac{\mu}{\id} * \sigma_{a},  
$
the identity (58) in \cite{KS}
\begin{equation}                                                              \label{lem211}
 \sum_{n\leq x}\frac{(\id_{1+a}*\phi)(n)}{n} 
=\frac{\zeta(1-a)}{\zeta(2)}x + \frac{\zeta(1+a)}{(1+a)\zeta(2+a)}x^{1+a} + E_{a}(x), 
\end{equation}
where  
\begin{equation}                                                                \label{Ea}    
E_{a}(x) = \sum_{n\leq x}\frac{\mu(n)}{n}\Delta_{a}\left(\frac{x}{n}\right) + O_{a}(1),
\end{equation} 
and the  partial summation, we get  
\begin{align}                                                                  \label{lem212}
& \sum_{n\leq x}\frac{(\id_{1+a}*\phi)(n)}{n}\log\frac{n}{e} \nonumber \\
&=\frac{\zeta(1-a)}{\zeta(2)}x\log^{}x - \frac{2\zeta(1-a)}{\zeta(2)}x  + \frac{\zeta(1+a)}{(1+a)\zeta(2+a)}x^{1+a}\log x  \nonumber \\
&- \frac{(2+a)\zeta(1+a)}{(1+a)^{2}\zeta(2+a)}x^{1+a} +  E_{a}(x)\log\frac{x}{e} - \int_{1}^{x}\frac{E_{a}(u)}{u}du + O_{a}\left(1\right). 
\end{align}
 Using \eqref{Voronoia}, we obtain that
\begin{eqnarray}                                                 
\label{lem213}
\int_{1}^{x}\frac{E_{a}(u)}{u}du 
 &=&   \sum_{n\leq x}\frac{\mu(n)}{n}\int_{n}^{x}\Delta_{a}\left(\frac{u}{n}\right)\frac{du}{u} + O_{a}\left(\log x\right)   \nonumber \\
&=& O_{a}\left(\log x \right). 
\end{eqnarray}
From \eqref{lem212} and \eqref{lem213}, we complete  the proof of \eqref{lem210}. \\

By the fact that 
$
\frac{\phi_{1+a}*\phi}{\id}= \frac{\phi_{1+a}*\mu}{\id}*\1 
$
and the following identity, see \cite[Eq.~(60)]{KS}, 
\begin{align}                                                               
\label{lem221}
\sum_{n\leq x} \frac{(\phi_{1+a}*\phi)(n)}{n} 
&=\frac{\zeta(1-a)}{\zeta^{2}(2)}x  + \frac{\zeta(1+a)}{(1+a)\zeta^{2}(2+a)} x^{1+a}  + P_{a}(x), 
\end{align}
where  
\begin{align}                                                                         \label{Pa} 
P_{a}(x) = \sum_{n\leq x}\frac{(\mu*\mu)(n)}{n}\Delta_{a}\left(\frac{x}{n}\right) 
+ O_{a}\left((\log x)^2\right),
\end{align} 
and the partial summation, we get 
\begin{multline}                                                                 
\label{lem222}
 \sum_{n\leq x} \frac{(\phi_{1+a}*\phi)(n)}{n}\log\frac{n}{e} 
=\frac{\zeta(1-a)}{\zeta^{2}(2)}x\log^{}x 
 + \frac{\zeta(1+a)}{(1+a)\zeta^{2}(2+a)}x^{1+a}\log x \\- \frac{2\zeta(1-a)}{\zeta^{2}(2)}x 
- \frac{(2+a)\zeta(1+a)}{(1+a)^{2}\zeta^{2}(2+a)} x^{1+a} 
 + P_{a}(x)\log\frac{x}{e} - \int_{1}^{x}\frac{P_{a}(u)}{u}du + O_{a}\left(1\right). 
\end{multline}
We use  \eqref{Voronoia} to deduce the estimate     
\begin{eqnarray}                                                               \label{lem223}
\int_{1}^{x}\frac{P_{a}(u)}{u}du 
&= &   \sum_{n\leq x}\frac{(\mu*\mu)(n)}{n}\int_{1}^{x/n}\Delta_{a}\left(u\right)\frac{du}{u}  
+ O_{a}\left((\log x)^3\right) \nonumber \\
&=& O_{a}\left((\log x)^3\right), 
\end{eqnarray}
From \eqref{lem222}  and \eqref{lem223}, we complete  the proof of \eqref{lem220}. 
\end{proof}


\section{Proof of Theorems~\ref{th4}  and  \ref{th5}}


\subsection{Proof of Theorem~\ref{th4}}


We take $f=\id$ into \eqref{sum_thL} to get 
\begin{eqnarray}                                              
\label{cor11-1}
L(x;\id) 
&= &\sum_{n\leq x} \frac{(\id*\phi)(n)}{n}  \log \frac{n}{e} 
+  \frac{1}{2} \sum_{n\leq x}\frac{(\id*\Lambda)(n)}{n} \nonumber  \\
&+& \log\sqrt{2\pi} \sum_{n\leq x} 1          
+  {\Theta} \sum_{n\leq x}\frac{(\id* \phi_{-1})(n)}{n}   \nonumber \\
&:=& I_{1,1}+ I_{1,2} +I_{1,3} +I_{1,4}, 
\end{eqnarray}
say.  It follows that 
\begin{equation}                                              \label{coro11-2}
I_{1,2}= \frac12 \sum_{dl\leq x}\frac{\Lambda(l)}{l} = - \frac{\zeta'(2)}{2\zeta(2)} x + O(\log x),  
\end{equation}
and 
\begin{equation}                                              \label{coro11-3}
I_{1,3}=  \log\sqrt{2\pi} \ x+ O(1). 
\end{equation}
Since the identities
$
\frac{\id*\phi_{-1}}{\id}=\frac{\phi_{-1}}{\id}*{\bf 1}, 
$
 $\phi_{-1}(n)=O(\sigma_{-1}(n))$ for any positive integer $n$
 and   the estimate 
\begin{align}                                                                    \label{sigma}
\sum_{n\leq x}\frac{\sigma_{-1}(n)}{n} = O(\log x).
\end{align}
Then, we  have  
\begin{eqnarray*}                                          
\sum_{n\leq x}\frac{(\id* \phi_{-1})(n)}{n}  
&=& x \sum_{l\leq x}\frac{\phi_{-1}(l)}{l^2} + O\left(\sum_{l\leq x}\frac{\sigma_{-1}(l)}{l}\right)  \nonumber \\
&=& \frac{\zeta(3)}{\zeta(2)}x + O\left(x \sum_{l>  x}\frac{\sigma_{-1}(l)}{l^2}\right) 
 + O\left(\log x\right)  \nonumber \\ 
&=& \frac{\zeta(3)}{\zeta(2)}x + O\left(\log x\right).  
\end{eqnarray*}
Thus
\begin{align}                                                             \label{coro11-4} 
I_{1,4}={\Theta} \frac{\zeta(3)}{\zeta(2)}x + O\left(\log x\right).  
\end{align} 
Substituting \eqref{coro11-2}, \eqref{coro11-3}, \eqref{coro11-4}, \eqref{lem110} and \eqref{M}  into \eqref{cor11-1}, 
 we obtain  
\begin{align*}                                              
L^{}(x;\id) 
&= \frac{1}{\zeta(2)}x(\log x)^2  + \frac{1}{\zeta(2)}\left(2\gamma - 3 -\frac{\zeta'(2)}{\zeta(2)}\right)x\log x 
\nonumber \\
&- \frac{1}{\zeta(2)}\left(4\gamma -3 -2\frac{\zeta'(2)}{\zeta(2)} + \frac{\zeta'(2)}{2} - {\Theta\zeta(3)}\right)x 
+ x \log\sqrt{2\pi} \nonumber \\ 
& + \sum_{n\leq x}\frac{\mu(n)}{n}\Delta\left(\frac{x}{n}\right)\log\frac{x}{e} + O\left((\log x)^2\right),    
\end{align*}
which  completes   the proof  of \eqref{sum_cor11}. \\

Next,  we take $f=\phi$ into \eqref{sum_thL} to get 
\begin{align}                                                 
\label{coro12-1}
L^{}(x;\phi) 
&= \sum_{n\leq x} \frac{(\phi*\phi)(n)}{n}  \log \frac{n}{e} 
+  \frac{1}{2} \sum_{n\leq x}\frac{(\phi*\Lambda)(n)}{n}   \nonumber \\
&+ \log\sqrt{2\pi} \sum_{n\leq x} \frac{\phi(n)}{n}          
+  {\Theta} \sum_{n\leq x}\frac{(\phi* \phi_{-1})(n)}{n}   \nonumber \\
&:= I_{2,1}+ I_{2,2} +I_{2,3} +I_{2,4}, 
\end{align}
say.  
Using the formula (2.2) in \cite{K1};   
$$
\sum_{n\leq x}\frac{\phi(n)}{n} = \frac{x}{\zeta(2)} + O\left((\log x)^{2/3}(\log\log x)^{4/3}\right),
$$
 we have 
\begin{align}                                                                   \label{coro12-3}
I_{2,3}=   \frac{\log\sqrt{2\pi}}{\zeta(2)} x + O\left((\log x)^{2/3}(\log\log x)^{4/3}\right). 
\end{align}
For $I_{2,2}$, we notice that 
\begin{eqnarray}                                                  \label{coro12-2}
I_{2,2}
&=& \frac{x}{2\zeta(2)}\sum_{n\leq x}\frac{\Lambda(n)}{n^2}  
+ O\left((\log x)^{2/3}(\log\log x)^{4/3}\sum_{n\leq x}\frac{\Lambda(n)}{n}\right) \nonumber \\
&=& \frac{x}{2\zeta(2)}\left(-\frac{\zeta'(2)}{\zeta(2)} 
+ O\left(\frac{\log x}{x}\right)\right)  
+ O\left((\log x)^{5/3}(\log\log x)^{4/3}\right) \nonumber \\
&=& - \frac{\zeta'(2)}{2\zeta^{2}(2)} x  + O\left((\log x)^{5/3}(\log\log x)^{4/3}\right).   
\end{eqnarray}
As for $I_{2,4}$, one can write
\begin{eqnarray*}                  
\sum_{n\leq x}\frac{(\phi* \phi_{-1})(n)}{n}  
&=& \sum_{l\leq x}\frac{\phi_{-1}(l)}{l}\left(\frac{1}{\zeta(2)}\frac{x}{l} 
+ O\left((\log x)^{2/3}(\log\log x)^{4/3}\right)\right)   \\
&=& \frac{x}{\zeta(2)}\sum_{d\leq x}\frac{\phi_{-1}(d)}{d^2} 
+ O\left((\log x)^{2/3}(\log\log x)^{4/3}\sum_{l\leq  x}\frac{\sigma_{-1}(l)}{l}\right) \\ 
&=& \frac{\zeta(3)}{\zeta^{2}(2)} x + O\left((\log x)^{5/3}(\log\log x)^{4/3}\right).  
\end{eqnarray*}
It follows that  
\begin{align}                                                 
\label{coro12-4}
I_{2,4}={\Theta}\frac{\zeta(3)}{\zeta^{2}(2)}x  + O\left((\log x)^{5/3}(\log\log x)^{4/3}\right).  
\end{align}
Substituting \eqref{coro12-3}--\eqref{coro12-4}, \eqref{lem120} and \eqref{wideP} into  \eqref{coro12-1},  we conclude that  
\begin{align*}                                               
L^{}(x;\phi) 
&= \frac{1}{\zeta^{2}(2)}x(\log x)^2 
+ \frac{1}{\zeta^{2}(2)}\left(2\gamma - 3 -2\frac{\zeta'(2)}{\zeta(2)}\right)x\log x \nonumber \\
&- \frac{1}{\zeta^{2}(2)}\left(4\gamma -3 -4\frac{\zeta'(2)}{\zeta(2)}  + \frac{\zeta'(2)}{2} 
- {\Theta\zeta(3)}\right)x 
+ \frac{\log\sqrt{2\pi}}{\zeta(2)}x \nonumber \\
& + \sum_{n\leq x}\frac{(\mu*\mu)(n)}{n}\Delta\left(\frac{x}{n}\right)\log\frac{x}{e} 
+ O\left((\log x)^3\right).    
\end{align*}
Therefore, the desired result is proved.


\subsection{Proof of Theorem~\ref{th5}}


We take $f=\id_{1+a}$ into \eqref{sum_thL}  to get 
\begin{eqnarray}                                              
\label{coro21-1}
L(x;\id_{1+a}) 
&= &\sum_{n\leq x} \frac{(\id_{1+a}*\phi)(n)}{n}  \log \frac{n}{e} 
+  \frac{1}{2} \sum_{n\leq x}\frac{(\id_{1+a}*\Lambda)(n)}{n} \nonumber  \\
& +& \log\sqrt{2\pi} \sum_{n\leq x} n^{a}          
+  {\Theta} \sum_{n\leq x}\frac{(\id_{1+a}* \phi_{-1})(n)}{n}   \nonumber \\
&:=& J_{1,1}+ J_{1,2} +J_{1,3} +J_{1,4}, 
\end{eqnarray}
say. 
From the formula (53) in \cite{KS} 
\begin{align}                                                                          
\label{VV}
\sum_{n\leq x}n^a =\frac{x^{1+a}}{1+a} + \zeta(-a) + O_{a}\left(x^a \right),
\end{align}
with $-1<a<0$,  we have 
\begin{equation}                                                                      
\label{coro21-3}
J_{1,3}=  \frac{\log\sqrt{2\pi}}{1+a}x^{1+a} + \zeta(-a)\log\sqrt{2\pi} + O_{a}(x^a),  
\end{equation} 
and 
\begin{align}                                                                         
\label{coro21-2}
J_{1,2}
&= \frac12 \sum_{l\leq x}\frac{\Lambda(l)}{l}\sum_{d\leq x/l}d^a  \nonumber \\ 
&= - \frac{\zeta'(2+a)}{2(1+a)\zeta(2+a)} x^{1+a} + O_{a}(\log x).  
\end{align}
Since the identity 
$
\frac{\id_{1+a}*\phi_{-1}}{\id}= \frac{\phi_{-1}}{\id}*{\id_{a}}  
$
and the estimates 
$$
\sum_{n\leq x}\frac{\phi_{-1}(n)}{n} = O(1),  \qquad
 \sum_{n\leq x}\frac{\sigma_{-1}(n)}{n^{1+a}} = O_{a}(x^{-a}),  
$$
we get  
\begin{eqnarray*}                                          
\sum_{n\leq x}\frac{(\id_{1+a}* \phi_{-1})(n)}{n} 
&=& \frac{x^{1+a}}{1+a} \sum_{l\leq x}\frac{\phi_{-1}(l)}{l^{2+a}} 
+ \zeta(-a)\sum_{d\leq x}\frac{\phi_{-1}(d)}{d} + 
O_{a}\left(x^{a}\sum_{l\leq x}\frac{\sigma_{-1}(l)}{l^{1+a}}\right) \\
&=& \frac{\zeta(3+a)}{(1+a)\zeta(2+a)}x^{1+a} + O_{a}(1).   
\end{eqnarray*}
Thus
\begin{align}                                                                       \label{coro21-4} 
J_{1,4}= {\Theta} \frac{\zeta(3+a)}{(1+a)\zeta(2+a)}x^{1+a}  + O_{a}\left(1\right).  
\end{align}  
On substituting \eqref{coro21-3}--\eqref{coro21-4}, \eqref{lem210} and \eqref{Ma} into \eqref{coro21-1},  we obtain  
\begin{align*}                                              
L^{}(x;\id_{1+a}) 
&= \frac{\zeta(1-a)}{\zeta(2)}x\log^{}x 
-  \frac{2\zeta(1-a)}{\zeta(2)} x 
+  \frac{\zeta(1+a)}{(1+a)\zeta(2+a)} x^{1+a}\log x  \nonumber \\
&- \frac{1}{(1+a)\zeta(2+a)}\left(\frac{(2+a)\zeta(1+a)}{1+a} 
                                + \frac{\zeta'(2+a)}{2} 
 - {\Theta\zeta(3+a)}\right)x^{1+a}         \nonumber \\
& + \frac{\log\sqrt{2\pi}}{1+a}x^{1+a} 
+  \sum_{n\leq x}\frac{\mu(n)}{n}\Delta_{a}\left(\frac{x}{n}\right)\log\frac{x}{e} 
+  O_{a}\left(\log x\right),    
\end{align*}
which  completes   the proof  of \eqref{sum_cor21}. \\

Next,  we take $f=\phi$ into \eqref{sum_thL}  to get 
\begin{align}                                                 
\label{coro22-1}
L^{}(x;\phi_{1+a}) 
&= \sum_{n\leq x} \frac{(\phi_{1+a}*\phi)(n)}{n}  \log \frac{n}{e} 
+  \frac{1}{2} \sum_{n\leq x}\frac{(\phi_{1+a}*\Lambda)(n)}{n}   \nonumber \\
&+ \log\sqrt{2\pi} \sum_{n\leq x} \frac{\phi_{1+a}(n)}{n}          
+  {\Theta} \sum_{n\leq x}\frac{(\phi_{1+a}* \phi_{-1})(n)}{n}   \nonumber \\
&:= J_{2,1}+ J_{2,2} + J_{2,3} + J_{2,4}, 
\end{align}
say.  
From \eqref{VV}, we have 
\begin{eqnarray}                                                    
\label{UU}
\sum_{n\leq x}\frac{\phi_{1+a}(n)}{n} 
&=&  \sum_{l\leq x} \frac{\mu(l)}{l}\left( \frac{1}{(1+a)}\left(\frac{x}{l}\right)^{1+a}+\zeta(-a)+O_a\left(\left(\frac{x}{l}\right)^{a}\right)\right) \nonumber\\
&=& \frac{1}{(1+a)\zeta(2+a)}x^{1+a} + O_{a}(1).
\end{eqnarray}
It follows that 
\begin{align} 
\label{coro22-3}
J_{2,3} =   
\frac{\log\sqrt{2\pi}}{(1+a)\zeta(2+a)} x^{1+a} + O_{a}(1), 
\end{align}
and that  
\begin{eqnarray}                                                  \label{coro22-2}
J_{2,2}
&=& \frac{x^{1+a}}{2(1+a)\zeta(2+a)}\sum_{l\leq x}\frac{\Lambda(l)}{l^{2+a}}  
+ O_{a}\left(\log x\right)                                                     \nonumber \\
&=& - \frac{\zeta'(2+a)}{2(1+a)\zeta^2({2+a})} x^{1+a}  + O_{a}\left(\log x\right).   
\end{eqnarray}
Using \eqref{UU} and the formula 
$$
\sum_{n\leq x}\frac{\phi_{-1}(d)}{d^{2+a}} =
\frac{\zeta(3+a)}{\zeta(2+a)} + O\left(x^{-1-a}\right), 
$$
we can check 
\begin{align*}                  
   \sum_{n\leq x}\frac{(\phi_{1+a}* \phi_{-1})(n)}{n}  
&= \frac{\zeta(3+a)}{(1+a)\zeta^{2}(2+a)} x^{1+a} + O_{a}\left(\log x\right).  
\end{align*}
Thus    
\begin{align}                                                  \label{coro22-4}
J_{2,4}= {\Theta}\frac{\zeta(3+a)}{(1+a)\zeta^{2}(2+a)} x^{1+a} + O_{a}\left(\log x\right). 
\end{align}
Putting everything together and \eqref{lem220},  we conclude that  
\begin{align*}                                               
L^{}& (x;\phi_{1+a}) 
=  \frac{\zeta(1-a)}{\zeta^2(2)}x\log^{}x 
-  \frac{2\zeta(1-a)}{\zeta^2(2)} x 
+  \frac{\zeta(1+a)}{(1+a)\zeta^2(2+a)} x^{1+a}\log x  \nonumber \\
&- \frac{1}{(1+a)\zeta^2(2+a)}\left(\frac{(2+a)\zeta(1+a)}{(1+a)}   + \frac{\zeta'(2+a)}{2} 
 - {\Theta\zeta(3+a)}\right)x^{1+a}         \nonumber \\
& + \frac{\log\sqrt{2\pi}}{(1+a)\zeta(2+a)}x^{1+a} 
+  \sum_{n\leq x}\frac{(\mu*\mu)(n)}{n}\Delta_{a}\left(\frac{x}{n}\right)\log\frac{x}{e} 
+  O_{a}\left((\log x)^3\right).    
\end{align*}
 Therefore, the formula \eqref{sum_cor22} is proved. 


\section{Proof of Theorems~\ref{th0} and \ref{corol0}} 

\subsection{Proof of Theorem~\ref{th0}}

We recall that the Stirling formula, see~\cite[Eq. (2.8)]{Ni}, is given by  
\begin{equation}                                                                      
\label{Nielsen}
 L(l):= \sum_{m=1}^{l}\log m 
=l\log l - l +\frac12 \log l + \log\sqrt{2\pi} + \frac{\vartheta}{12l}
\end{equation}
for any positive integer $l$ and  $\vartheta$ is an  absolute constant satisfying $0<\vartheta<1$.
Using \eqref{KMU} and  \eqref{Nielsen}, 
there is a  certain  positive  constant   $\Theta$  such that   
\begin{eqnarray}                                                                              \label{KKK}   
K(x;f,g) 
&=&\sum_{k\leq x}\frac{(f\cdot \log * g\cdot \id)(k)}{k} 
+ \sum_{k\leq x}\frac{(f*g\cdot L)(k)}{k}  \nonumber \\
&=& \sum_{dl\leq x}\frac{f(d)\log d}{d}  g(l)   
+  \sum_{dl\leq x}\frac{f(d)}{d}  g(l)  \log l 
-  \sum_{dl\leq x}\frac{f(d)}{d}  g(l)                                \\ 
&&+ \frac12 \sum_{dl\leq x}\frac{f(d)}{d}\frac{g(l)\log l}{l}  
 + \log\sqrt{2\pi} \sum_{dl\leq x}\frac{f(d)}{d}\frac{g(l)}{l} 
 + \Theta \sum_{dl\leq x}\frac{f(d)}{d}\frac{g(l)}{l^2}.             \nonumber   
\end{eqnarray}
Using properties of the Dirichlet convolution, we get the formula~\eqref{sum_thK}. Which completes the proof. 
\subsection{Proof of Theorem~\ref{corol0}}
We take $f=\id$ and $g=\mu$ into  \eqref{sum_thK} to get 
$$
K(x;\id,\mu) = \sum_{k\leq x}\frac{1}{k}\sum_{j=1}^{k}c_{k}(j)\log j,
$$
and  
\begin{multline*}                                                                       
 K(x;\id,\mu) 
=  \sum_{n\leq x}\frac{(\id * \id \cdot \mu)(n)}{n} \log \frac{n}{e} 
+  \frac{1}{2} \sum_{n\leq x}\frac{(\id * \mu \cdot \log )(n)}{n}        \\  
+ \log\sqrt{2\pi} \sum_{n\leq x}\frac{(\id *\mu)(n)}{n}          
+  {\Theta} \sum_{n\leq x}\frac{(\id * \mu \cdot \id_{-1})(n)}{n}. 
\end{multline*}
It follows that 
\begin{multline}                                                                            \label{QQ}
K(x;\id,\mu) 
=  \sum_{n\leq x} \sum_{d|n}\mu(d)\log \frac{n}{e}
+  \frac{1}{2} \sum_{d\leq x}\frac{\mu(d)\log d}{d}\sum_{l\leq x/d}1 \\
+ \log\sqrt{2\pi} \sum_{n\leq x} \frac{\phi(n)}{n} +  {\Theta} \sum_{d\leq x}\frac{\mu(d)}{d^2}\sum_{l\leq x/d}1.  
\end{multline}
Since $\sum_{d|n}\mu(d)\log n = 0$ for any positive integer  $n$. Then, we have 
$$
\sum_{n\leq x} \sum_{d|n}\mu(d)\log \frac{n}{e} = -1.
$$
By \cite[Chapter I and VI (Sections 8 and 9)]{SMC}, we write  
$$
\log\sqrt{2\pi} \sum_{n\leq x} \frac{\phi(n)}{n}  = \frac{\log\sqrt{2\pi}}{\zeta(2)} x +  O\left((\log x)^{2/3}(\log \log x)^{4/3}\right), 
$$ 
\begin{align*}
\frac{1}{2} \sum_{d\leq x}\frac{\mu(d)\log d}{d}\sum_{l\leq x/d}1 
&=\frac{x}{2} \sum_{d\leq x}\frac{\mu(d)}{d^2}\log d + O\left(\sum_{d\leq x}\frac{\log d}{d}\right)  \\
&=\frac{\zeta'(2)}{2\zeta^{2}(2)}x + O\left(\log^{2}x\right),
\end{align*}
and 
\begin{align*}
{\Theta} \sum_{d\leq x}\frac{\mu(d)}{d^2}\sum_{l\leq x/d}1  
&=\frac{\Theta}{\zeta(3)} x + O(1).
\end{align*} 
Putting everything together we obtain the formula \eqref{sum_coroK}.
\begin{rem}                                                                                  \label{rem5}
Notice that the formula \eqref{QQ} can be derived by \eqref{toth}, \eqref{Nielsen} and the identity $\mu\cdot \log * \1 = -\Lambda$. That is   
\begin{equation}                                                                       \label{sumaia}
    \frac{1}{k} \sum_{j=1}^{k}c_{k}(j)\log j  
=  \frac{1}{2}\sum_{d|k}\frac{\mu(d)}{d}\log d + \log\sqrt{2\pi} \sum_{d|k}\frac{\mu(d)}{d} 
+ \frac{\vartheta}{12}\sum_{d|k}\frac{\mu(d)}{d^2}.
\end{equation}
That means 
\begin{multline*}                                                                           
 K(x;\id,\mu) 
= \frac{1}{2} \sum_{d\leq x}\frac{\mu(d)\log d}{d}\sum_{l\leq x/d}1    \\
 \qquad + \log\sqrt{2\pi} \sum_{n\leq x} \frac{\phi(n)}{n} +  {\Theta} \sum_{d\leq x}\frac{\mu(d)}{d^2}\sum_{l\leq x/d}1 
+ O(1).   
\end{multline*}
\end{rem}


\section{Proof of Corollaries \ref{corol2} and  \ref{th9}}



\subsection{Proof of  Corollary  \ref{corol2}} 


For $f=g=\1$, one can easy to see that 
$$
s_{k}(j) = \sum_{d|\gcd(k,j)}\1(d)\1\left(\frac{k}{d}\right) = \tau(\gcd(k,j)).
$$
We   use  \eqref{sum_thKL1}  to  obtain     
\begin{align}                                                                            \label{sum_thKL111}
& \sum_{k\leq x}\frac{1}{k}\sum_{j=1}^{k}\tau(\gcd(k,j)) \log j \\ 
&= \sum_{n\leq x} \frac{\sigma(n)}{n}  \log \frac{n}{e} 
+  \frac{1}{2} \sum_{n\leq x}\frac{l(n)}{n}  
 + \log\sqrt{2\pi} \sum_{n\leq x} \frac{\tau(n)}{n}       
+  {\Theta} \sum_{n\leq x}\frac{\sigma_{-1}(n)}{n}.    \nonumber 
\end{align}
Using the well-known formula  
$$
\sum_{n\leq x}\frac{\log n}{n} = \frac{1}{2}(\log x)^2 + A + O\left(\frac{\log x}{x}\right),
$$
with $A$ being a constant, and the partial summation, we get
\begin{align}                                                                              \label{A}
   \sum_{n\leq x}\frac{l(n)}{n}
&= \sum_{l\leq x}\frac{\log l}{l}\sum_{d\leq x/l}\frac{1}{d}  \nonumber \\
&= (\log x +\gamma)\sum_{l\leq x}\frac{\log l}{l} - \sum_{l\leq x}\frac{(\log l)^2}{l} + O\left(\log x\right)  \nonumber \\
&= \frac{1}{6}(\log x)^3 + \frac{\gamma}{2} (\log x)^2 + O\left(\log x\right).
\end{align}
By the following formula, see~\cite[Eq. (2.14)]{K},    
$$
\sum_{n\leq x}\frac{\sigma(n)}{n}=\zeta(2)x -\frac{1}{2}\log x +O((\log x)^{2/3}), 
$$
and the partial summation, we find that 
\begin{eqnarray}                                                                            
\label{B22}
\sum_{n\leq x}\frac{\sigma(n)}{n}\log\frac{n}{e}    
&=& \left(\zeta(2) x - \frac{1}{2}\log x + O\left((\log x)^{2/3}\right)\right)\log\frac{x}{e} \nonumber \\
&&- \int_{1}^{x}\left(\zeta(2)u-\frac{1}{2}\log u + O\left((\log u)^{2/3}\right)\right)\frac{du}{u} \nonumber \\
&=& \zeta(2)x\log x - 2\zeta(2)x -\frac{1}{4} (\log x)^2 + O\left((\log x)^{5/3}\right).
\end{eqnarray}
From \eqref{sigma}, \eqref{A}, \eqref{B22},  
and 
$$
\sum_{n\leq x}\frac{\tau(n)}{n}= \frac{1}{2}(\log x)^2 + 2\gamma \log x + O(1),
$$
(see~\cite[Eq. (2.20)]{K}),  we therefore deduce \eqref{sum_thKL12}.


\subsection{Proof of  Corollary  \ref{th9}}  


We recall that 
\begin{equation*}
U_{f, g}(s)=-F'(s)G(s-1)+F(s)G_L(s),
\end{equation*}
where 
\begin{equation*}
G_L(s)=\sum_{k=1}^{\infty}\frac{g(k)L(k)}{k^s}, 
\end{equation*}
and $L(k)=\sum_{m=1}^{k}\log m.$ Using \eqref{Nielsen}, we immediately get the formula 
\begin{multline*}                                                                              
 U_{f,g}(s)  
=-F'(s)G(s-1) - F(s)G'(s-1) - F(s)G(s-1)                                                      \\
 -\frac12 F(s)G'(s) + \log\sqrt{2\pi} F(s)G(s) + \Theta F(s)G(s+1). 
\end{multline*}
Again, we have 
\begin{multline*}                                                                                
 U_{f,g}(s)  
= \sum_{d=1}^{\infty} \frac{f(d)\log d}{d^s} \sum_{l=1}^{\infty}\frac{g(l)}{l^{s-1}}    
+  \sum_{d=1}^{\infty}\frac{f(d)}{d^s} \sum_{l=1}^{\infty}\frac{g(l)\log l}{l^{s-1}}             
-  \sum_{d=1}^{\infty}\frac{f(d)}{d^s} \sum_{l=1}^{\infty}\frac{g(l)}{l^{s-1}}   
\\+ \frac12 \sum_{d=1}^{\infty} \frac{f(d)}{d^s} \sum_{l=1}^{\infty}\frac{g(l)\log l}{l^s}     
 + \log\sqrt{2\pi} \sum_{d=1}^{\infty} \frac{f(d)}{d^s} \sum_{l=1}^{\infty}\frac{g(l)}{l^s} 
 + \Theta \sum_{d=1}^{\infty} \frac{f(d)}{d^s} \sum_{l=1}^{\infty}\frac{g(l)}{l^{s+1}}             
\end{multline*}
From this latter with $f$ replaced by  $f*\mu$ and $g=\1$ and using the identity  
$$
 \sum_{d=1}^{\infty} \frac{(f*\mu)(d)\log d}{d^s}\zeta(s-1)= \frac{\left(F'(s)\zeta(s)-F(s)\zeta'(s)\right)\zeta(s-1)}{\zeta^{2}(s)},
$$
we obtain 
\begin{multline*}                                                                              
U_{f,g}(s)                                       
= \frac{(F(s)\zeta'(s)- F'(s)\zeta(s))\zeta(s-1)}{\zeta^{2}(s)} 
- \left(\zeta(s-1) + \zeta'(s-1)\right)\frac{F(s)}{\zeta(s)}\\
 -  \frac{F(s)}{2}\frac{\zeta'(s)}{\zeta(s)}  + \log\sqrt{2\pi} F(s)  + \Theta \frac{F(s)}{\zeta(s)}{\zeta(s+1)}.
\end{multline*}
This completes the proof.

%
%
%
%

\bigskip


\section*{Acknowledgement}
The second  author is supported by the Austrian Science
Fund (FWF) : Project F5507-N26, which is part
of the special Research Program  `` Quasi Monte
Carlo Methods : Theory and Application''.


\medskip\noindent {\footnotesize Isao Kiuchi: Department of Mathematical Sciences, Faculty of Science,
Yamaguchi University, Yoshida 1677-1, Yamaguchi 753-8512, Japan. \\
e-mail: {\tt kiuchi@yamaguchi-u.ac.jp}}

\medskip\noindent {\footnotesize Sumaia Saad Eddin: 
Institute of Financial Mathematics and Applied Number Theory, Johannes Kepler University, Altenbergerstrasse 69, 4040 Linz, Austria.\\
e-mail: {\tt sumaia.saad\_eddin@jku.at}}

\end{document}